\documentclass[a4paper]{article}

\usepackage[round,comma]{natbib}
\usepackage{amsmath,arcs,amsthm}
\usepackage[table,svgnames,x11names]{xcolor}
\usepackage{graphicx,wasysym}
\usepackage[T1]{fontenc}
\usepackage[adobe-utopia]{mathdesign}

\usepackage{a4wide}

\def\Tr{\text{\rm trace}}
\def\sff{{\sf f}}
\def\sfh{{\sf h}}
\def\sfg{{\sf g}}
\def\sfl{{\sf l}}

\def\hatthetak{\widehat{\hspace{0.4pt}\theta}\hspace{-1.3pt}_\bk}
\def\gmin{g_{\min}}
\def\bX{\boldsymbol{X}}

\def\ds{\displaystyle}

\def\1{\mathbf 1}
\def\RR{\mathbb R}
\def\PP{\mathbb P}
\def\NN{\mathbb N}
\def\ZZ{\mathbb Z}

\def\by{\boldsymbol y}
\def\bx{\boldsymbol x}

\def\oomega{{\boldsymbol\omega}}

\def\E{\mathbb E}
\def\Pb{\mathbf P}
\def\Ex{\mathbf E}

\def\bk{{\boldsymbol k}}

\def\T{\top}

\def\KL{\mathcal K}

\def\bQ{\boldsymbol Q}
\def\ds{\displaystyle}
\def\zb{z_\gamma}
\def\zbstar{z_{\gamma^*}}

\newtheorem{proposition}{Proposition}
\newtheorem{remark}{Remark}
\newtheorem{theorem}{Theorem}
\newtheorem{lemma}{Lemma}

\begin{document}

\title{Tight conditions for consistent variable selection in high dimensional nonparametric regression}

\author{La\"etitia Comminges and Arnak S. Dalalyan\\
Universit\'e Paris Est/ ENPC\\
LIGM/IMAGINE\\
\texttt{\small laetitia.comminges,dalalyan@imagine.enpc.fr}}

\maketitle

\begin{abstract}
We address the issue of variable selection in the regression model with very high ambient dimension, \textit{i.e.}, 
when the number of covariates is very large. The main focus is on the situation where the number of relevant covariates, called 
intrinsic dimension, is much smaller than the ambient dimension. Without assuming any parametric form of the underlying 
regression function, we get tight conditions making it possible to consistently estimate the set of relevant variables. 
These conditions relate the intrinsic dimension to the ambient dimension and to the sample size.  The procedure that is provably 
consistent under these tight conditions is simple and is based on comparing the empirical Fourier coefficients with an 
appropriately chosen threshold value.
\end{abstract}

\parskip=3pt

\section{Introduction}\label{sec:1}

Real-world data such as those obtained from neuroscience, chemometrics, data mining, or sensor-rich 
environments are often extremely high-dimensional, severely underconstrained (few data samples compared 
to the dimensionality of the data), and interspersed with a large number of irrelevant or redundant 
features. Furthermore, in most situations the data is contaminated by noise making it even more difficult 
to retrieve useful information from the data. Relevant variable selection is a compelling approach for 
addressing statistical issues in the scenario of high-dimensional and noisy data with small sample size. 
Starting from \cite{Mallows}, \cite{Akaike,Schwarz} who introduced respectively the famous criteria $C_p$, AIC and 
BIC, the problem of variable selection has been extensively studied in the statistical
and machine learning literature both from the theoretical and algorithmic viewpoints. It appears, however, that 
the theoretical limits of performing variable selection in the context of nonparametric regression are still 
poorly understood, especially in the case where the ambient dimension of covariates, denoted by $d$, is much 
larger than the sample size $n$. The purpose of the present work is to explore this setting under the assumption
that the number of relevant covariates, hereafter called intrinsic dimension and denoted by $d^*$, may grow with 
the sample size but remains much smaller than the ambient dimension $d$.

In the important particular case of linear regression, the latter scenario has been the subject of a number of recent 
studies. Many of them rely on $\ell_1$-norm penalization (as for instance in \cite{Tibsh,Zhao,Meinshausen}) and
constitute an attractive alternative to iterative variable selection procedures proposed by \cite{alquier,Zhang09,Ting} 
and to marginal regression or correlation screening explored in \cite{Wass09,Fan09}. Promising results for feature selection 
are also obtained by minimax concave penalties in \cite{ZnangCH10}, by Bayesian approach in \cite{Scott10} 
and by higher criticism in \cite{Donoho09}. Extensions to other settings including logistic regression, generalized linear 
model and Ising model have been carried out in \cite{Bunea09,Ravik,Fan09}, respectively. Variable selection in the context 
of groups of variables with disjoint or overlapping groups has been studied by~\cite{svssin,Pontil,Obozinski}. Hierarchical 
procedures for  selection of relevant covariates have been proposed by~\cite{Bach09,Bickeletal} and \cite{BinYu09}.

It is now well understood that in the high-dimensional linear regression, if the Gram matrix satisfies some variant of 
irrepresentable condition, then consistent estimation of the pattern of relevant variables---also called the sparsity 
pattern---is possible under the condition $d^*\log (d/d^*)=o(n)$ as $n\to\infty$. Furthermore, it is well known that if 
$(d^*\log (d/d^*))/n$ remains bounded from below by some positive constant when $n\to\infty$, then it is impossible to 
consistently recover the sparsity pattern. Thus, a tight condition exists that describes in an exhaustive manner the
interplay between the quantities $d^*$, $d$ and $n$ that guarantees the existence of consistent estimators. The situation
is very different in the case of non-linear regression, since, to our knowledge, there is no result providing tight
conditions for consistent estimation of the sparsity pattern.  

The papers \cite{LaffertyWasserman} and \cite{BertinLecue}, closely related to the present work, consider the problem of 
variable selection in nonparametric Gaussian regression model. They prove the consistency of the proposed procedures under 
some assumptions that---in the light of the present work---turn out to be suboptimal. More precisely, in \cite{LaffertyWasserman}, the 
unknown regression function is assumed to be four times continuously differentiable with bounded derivatives. The algorithm 
they propose, termed Rodeo, is a greedy procedure performing simultaneously local bandwidth choice and variable selection. Under 
the assumption that the density of the sampling design is continuously differentiable and strictly positive, Rodeo is shown to 
converge when the ambient dimension $d$ is $O({\log n}/{\log \log n})$ while the intrinsic dimension $d^*$ does not increase 
with $n$. On the other hand, \cite{BertinLecue} propose a procedure based on the $\ell_1$-penalization of local polynomial 
estimators and prove its consistency when $d^*=O(1)$ but $d$ is allowed to be as large as $\log n$, up to a multiplicative 
constant. They also have a weaker assumption on the regression function which is merely assumed to belong to the Holder class 
with smoothness $\beta>1$. 

This brief review of the literature reveals that there is an important gap in consistency conditions for the linear regression
and for the non-linear one. For instance, if the intrinsic dimension $d^*$ is fixed, then the condition guaranteeing consistent
estimation of the sparsity pattern is $(\log d)/n\to 0$ in linear regression whereas it is $d=O(\log n)$ in the nonparametric case. 
While it is undeniable that the nonparametric regression is much more complex than the linear one, it is however not easy to find
a justification to such an important gap between two conditions. The situation is even worse in the case where $d^*\to\infty$. 
In fact, for the linear model with at most polynomially increasing ambient dimension $d=O(n^k)$, it is possible to estimate
the sparsity pattern for intrinsic dimensions $d^*$ as large as $n^{1-\epsilon}$, for some $\epsilon>0$. In other words, the 
sparsity index can be almost on the same order as the sample size. In contrast, in nonparametric regression, there is no 
procedure that is proved to converge to the true sparsity pattern when both $n$ and $d^*$ tend to infinity, even if $d^*$  
grows extremely slowly.

In the present work, we fill this gap by introducing a simple variable selection procedure that selects the relevant variables
by comparing some well chosen empirical Fourier coefficients to a prescribed significance level. Consistency of this procedure
is established under some conditions on the triplet $(d^*,d,n)$ and the tightness of these conditions is proved. The main take-away 
messages deduced from our results are the following:
\begin{enumerate}
\item[$\checked$] When the number of relevant covariates $d^*$ is fixed and the sample size $n$ tends to infinity, there exist 
positive real numbers ${c}_*$ and ${c}^*$ such that (a) if\/ $(\log d)/n\le {c}_*$ the estimator proposed in  Section~\ref{sec:3}
is consistent and (b) no estimator of the sparsity pattern may be consistent if\/ $(\log d)/n\ge {c}^*$.
\item[$\checked$] When the number of relevant covariates $d^*$ tends to infinity with $n\to\infty$, then there exist  
real numbers $\underline{c}_i$ and $\bar c_i$, $i=1,\ldots,4$ such that $\underline{c_i}>0$, $\bar c_i>0$ for $i=1,2,3$ 
and (a) if $\underline{c}_1d^*+\underline{c}_2\log d^*+\underline{c}_3\log\log d-\log n
<\underline{c}_4$ the estimator proposed in  Section~\ref{sec:3} is consistent  and (b) no estimator of the sparsity pattern may be consistent 
if $\bar c_1d^*+\bar c_2\log d^*+\bar c_3\log\log d-\log n >\bar c_4$. 
\item[$\checked$] In particular, if $d$ grows not faster than a polynomial in $n$, then there exist 
positive real numbers ${c}_0$ and ${c}^0$ such that (a) if $d^*\le {c}_0\log n$ the estimator proposed in  Section~\ref{sec:3}
is consistent and (b) no estimator of the sparsity pattern may be consistent if $d\ge {c}^0\log n$.
\end{enumerate} 
Very surprisingly, the derivation of these results required from us to apply some tools from complex analysis, such as the 
Jacobi $\theta$-function and the saddle point method, in order to evaluate the number of lattice points lying in a ball of an 
Euclidean space with increasing dimension. 

The rest of the paper is organized as follows. The notation and assumptions necessary for stating our main results are presented in
Section~\ref{sec:2}. In Section~\ref{sec:3}, an estimator of the set of relevant covariates is introduced and its consistency is established.  
The principal condition required in the consistency result involves the number of lattice points in a ball of a high-dimensional Euclidean 
space. An asymptotic equivalent for this number is obtained in Section~\ref{sec:4} via the Jacobi $\theta$-function and the saddle point 
method. Results on impossibility of consistent estimation of the sparsity pattern are derived in Section~\ref{sec:5}, while the relation
between consistency and inconsistency results are discussed in Section~\ref{sec:6}. The technical parts of the proofs are postponed to the 
Appendix.  

\section{Notation and assumptions}\label{sec:2}

We assume that $n$ independent and identically distributed pairs of input-output variables $(\bX_i,Y_i)$, $i=1,\ldots,n$ are 
observed that obey the regression model
$$
Y_i=\sff(\bX_i)+\sigma \varepsilon_i,\qquad i=1,\ldots,n.
$$
The input variables $\bX_1,\ldots,\bX_n$ are assumed to take values in $\RR^d$ while the output variables $Y_1,\ldots,Y_n$ 
are scalar. As usual, the noise $e_1,\ldots,e_n$ is such that $\Ex[\varepsilon_i|\bX_i]=0$, $i=1,\ldots,n$; some additional conditions 
will be imposed later. Without requiring from $\sff$ to be of a special parametric form, we aim at recovering 
the set $J\subset\{1,\ldots,d\}$ of its relevant covariates. 
 
It is clear that the estimation of $J$ cannot be accomplished without imposing some further assumptions on $\sff$ and the 
distribution $P_X$ of the input variables. Roughly speaking, we will assume that $\sff$ is differentiable with a squared 
integrable gradient and that $P_X$ admits a density which is bounded from below. More precisely,  let $\sfg$ denote the 
density of $P_X$ w.r.t.\ the Lebesgue measure. 
\begin{description}
\item[{[C1]}] We assume that $\sfg(\bx)=0$ for any $\bx\not\in[0,1]^d$ and that $\sfg(\bx)\ge{\gmin}$ for any $\bx\in[0,1]^d$.
\end{description}
To describe the smoothness assumption imposed on $\sff$, let us introduce the Fourier basis 
\begin{equation}
\varphi_\bk(\bx)=
\begin{cases}
1, & \bk=0 ,\\
\sqrt{2}\cos(2\pi\,\bk\cdot\bx), & \bk\in(\ZZ^d)_+,\\
\sqrt{2}\sin(2\pi\,\bk\cdot\bx), &-\bk\in(\ZZ^d)_+ ,
\end{cases}
\end{equation}
where $(\ZZ^d)_+$ denotes the set of all $\bk\in\ZZ^d\setminus\{0\}$ such that the first nonzero element of $\bk$ is positive and 
$\bk\cdot\bx$ stands for the the usual inner product in $\RR^d$. In what follows, we use the notation $\langle\cdot,\cdot\rangle$ 
for designing the scalar product in $L^2([0,1]^d;\RR)$, that is $\langle \sfh,\tilde \sfh\rangle=\int_{[0,1]^d} \sfh(\bx)\tilde \sfh(\bx)\,d\bx$ 
for every $\sfh,\tilde \sfh\in L^2([0,1]^d;\RR)$.  Using this orthonormal Fourier basis, we define
$$
\Sigma_L=\bigg\{f: \sum_{\bk\in\ZZ^d} k_j^{2}\langle \sff,\varphi_\bk\rangle^2\le L;\quad\forall j\in\{1,\ldots,d\} \bigg\}.
$$ 
To ease notation, we set $\theta_\bk[\sff]=\langle \sff,\varphi_\bk\rangle$ for all $\bk\in\ZZ^d$. In addition to the smoothness, we 
need also to require that the relevant covariates are sufficiently relevant for making their identification possible. This is done 
by means of the following condition.
\begin{description}
\item[{[C2$(\kappa,L)$]}] The regression function $\sff$ belongs to $\Sigma_L$. Furthermore, for some subset $J\subset\{1,\ldots,d\}$ of 
cardinality $\le d^*$, there exists a function $\bar\sff:\RR^{|J|}\to\RR$  such that 
$\sff(\bx) =\bar\sff(\bx_J)$, $\forall \bx\in\RR^d$ and it holds that 
\begin{equation}\label{ident}
Q_j[\sff]\triangleq\sum_{\bk : k_j\neq 0}\theta_\bk[\sff]^2\geq\kappa,\ \forall j\in J.
\end{equation} 
Hereafter, we will refer to $J$ as the sparsity pattern of $\sff$.
\end{description}
One easily checks that $Q_j[\sff]=0$ for every $j$ that does not lie in the sparsity pattern. This provides a characterization of the
sparsity pattern as the set of indices of nonzero coefficients of the vector $\bQ[\sff]=(Q_1[\sff],\ldots,Q_d[\sff])$.

The next assumptions imposed to the regression function and to the noise require their boundedness in an appropriate sense. 
These assumptions are needed in order to prove, by means of a concentration inequality, the closeness of the empirical 
coefficients to the true ones. 
\begin{description}
\item[{[C3$(L_\infty,L_2)$]}] The $L^\infty([0,1]^d,\RR,P_X)$ and $L^2([0,1]^d,\RR,P_X)$ norms of the function $\sff$ are 
bounded from above respectively by $L_\infty>0$ and $L_2$, \textit{i.e.}, 
$P_X\big({\bx\in[0,1]^d}: |\sff(\bx)|\le L_\infty\big)=1$ and $\int_{[0,1]^d} \sff(\bx)^2\sfg(\bx)\,d\bx\le L_2^2$.
\end{description}
\begin{description}
\item[{[C4]}] The noise variables satisfy a.e. $\Ex[e^{t \varepsilon_i}|\bX_i]\le e^{t^2/2}$ for all $t>0$.
\end{description}
\begin{remark}
The primary aim of this work is to understand when it is possible to estimate the sparsity pattern (with theoretical guarantees on
the convergence of the estimator) and when it is impossible. The estimator that we will define in the next section is intended to 
show the possibility of consistent estimation, rather than being a practical procedure for recovering the sparsity pattern. Therefore, 
the estimator will be allowed to depend on the parameters $\gmin$, $L$, $\kappa$ and $M$ appearing in conditions {\bf [C1-C3]}.
\end{remark}

\section{Consistent estimation of the set of relevant variables}\label{sec:3}

The estimator of the sparsity pattern $J$ that we are going to introduce now is based on the following simple observation: 
if $j\not\in J$ then $\theta_\bk[\sff]=0$ for every $\bk$ such that $k_j\not=0$. In contrast, if $j\in J$ then there exists
$\bk\in\ZZ^d$ with $k_j\not=0$ such that $|\theta_\bk[\sff]|>0$. To turn this observation into an estimator of $J$, we 
start by estimating the Fourier coefficients $\theta_\bk[\sff]$ by their empirical counterparts:
$$
\hatthetak=\frac{1}{n}\sum_{i=1}^n  \frac{\varphi_\bk(\bX_i)}{\sfg(\bX_i)}Y_i,\qquad \bk\in\ZZ^d.
$$
Then, for every $\ell\in\NN$ and for any $\gamma >0$, we introduce the notation  $S_{m,\ell}=\big\{\bk\in\ZZ^d:\ \|\bk\|_{2}\le m,\ \|\bk\|_0\le \ell \big\}$
and $N(d^*,\gamma)=\{\bk\in\ZZ^{d^*}: \|\bk\|_2^2\le \gamma d^*\,\&\, k_1\not=0\}$. Finally our estimator is defined by 
\begin{equation}\label{hatJ}
\widehat{J}_n(m,\lambda)=\Big\{j\in\{1,\ldots,d\}:\ \max_{\bk\in S_{m,d^*}:\,k_j\not=0} |\hatthetak|>{\lambda} \Big\},
\end{equation}
where $m$ and ${\lambda}$ are some parameters to be defined later. 
The notation $a\wedge b$, for two real numbers $a$ and $b$, stands for $\min(a,b)$. 

\begin{theorem}\label{thm1}
Let conditions {\bf [C1-C4]} be fulfilled with some known constants $\gmin,L,\kappa$ and $L_2$. 
Assume furthermore that the design density $\sfg$ and an upper estimate on the noise magnitude $\sigma$ are 
available. Set $m=(2Ld^*/\kappa)^{1/2}$ and $\lambda=4(\sigma+L_2)\big({d^*\log(6md)}/{n\gmin^2})^{1/2}$. If 
\begin{align}
\frac{L_\infty^2 d^*\log(6md)}{n}\leq L_2^2,\quad\text{and}\quad 
\frac{128({\sigma}+L_2)^2d^*N(d^*,2L/\kappa)\log(6md)}{n\gmin^2}&\leq \kappa,
 \label{cond3}
\end{align}
then the estimator $\widehat{J}(m,\lambda)$ satisfies $\Pb\big(\widehat{J}(m,\lambda)\neq J\big) \leq 3(6md)^{-d^*}$. 
\end{theorem}

If we take a look at the conditions of Theorem~\ref{thm1} ensuring the consistency of the estimator $\widehat J$, 
it becomes clear that the strongest requirement is the second inequality in (\ref{cond3}). To some extent, this 
condition requires that $(d^*N(d^*,2L/\kappa)\log d)/n$ is bounded from above by some constant. To further analyze 
the interplay between $d^*$, $d$ and $n$ implied by this condition, we need an equivalent to $N(d^*,2L/\kappa)$ as 
the intrinsic dimension $d^*$ tends to infinity. As proved in the next section, $N(d^*,2L/\kappa)$ diverges 
exponentially fast, making inequality (\ref{cond3}) impossible for $d^*$ larger than $\log n$ up to a multiplicative 
constant.

It is also worth stressing that although we require the $P_X$-a.e.\ boundedness of $\sff$ by some constant $L_\infty$, 
this constant is not needed for computing the estimator proposed in Theorem~\ref{thm1}. Only constants related to some quadratic
functionals of the sequence of Fourier coefficients $\theta_\bk[\sff]$ are involved in the tuning parameters $m$ and $\lambda$. This point might 
be important for designing practical estimators of $J$, since the estimation of quadratic functionals is more realistic, see for instance~\cite{LaurentMassart}, than the estimation of $\sup$-norm.  

The result stated above provides also a level of relevance $\kappa$ for the covariates of $\bX$ making their 
identification possible.  In fact, an alternative way of reading Theorem~\ref{thm1} is the following: 
if conditions \textbf{[C1-C4}] and ${L_\infty^2 d^*\log(6md)}\leq {n}L_2^2$ are fulfilled, then 
the estimator $\widehat J(m,\lambda)$---with arbitrary tuning parameters $m$ and $\lambda$---satisfies 
$\Pb(\widehat J(m,\lambda)\not=J)\le 3(6md)^{-d^*}$ provided that the smallest level of relevance $\kappa$ 
for components $X_j$ of $\bX$ with $j\in J$ is not smaller than $8\lambda^2 N(d^*,m^2/d^*)$.

\section{Counting lattice points in a ball}\label{sec:4}

The aim of the present section is to investigate the properties of the quantity $N(d^*,m^2/{d^*})$ that is
involved in the conditions ensuring the consistency of the proposed procedure. Quite surprisingly, the 
asymptotic behavior of $N(d^*,m^2/{d^*})$ turns out to be related to the Jacobi $\theta$-function. In order to
show this, let us introduce some notation. For a positive number $\gamma$, we set
\begin{equation*}
\mathcal{C}_1(d^*,\gamma)=\Big\{\bk \in \ZZ^{d^*} : k_1^{2}+...+k_{d^*}^{2}\leq \gamma d^* \Big\},\quad
\mathcal{C}_2(d^*,\gamma)=\Big\{\bk \in \ZZ^{d^*} :  k_2^{2}+...+k_{d^*}^{2}\leq \gamma d^*\ \&\ k_1=0\Big\}
\end{equation*}
along with $N_1(d^*,\gamma)=\text{Card}\mathcal{C}_1(d^*,\gamma)$ and  $N_2(d^*,\gamma)=\text{Card}\mathcal{C}_2(d^*,\gamma)$. 
In simple words, $N_1(d^*,\gamma)$ is the number of (integer) lattice points lying in the $d^*$-dimensional ball 
with radius $(\gamma d^*)^{1/2}$ and centered at the origin, while $N_2(d^*,\gamma)$ is the number of (integer) lattice 
points with the first coordinate equal to zero and lying in the $d^*$-dimensional ball with radius $(\gamma d^*)^{1/2}$ 
and centered at the origin.  With these notation, the quantity $N(d^*,2L/\kappa)$ of Theorem~\ref{thm1} 
can be written as $N_1(d^*,2 L/{\kappa})-N_2(d^*,2{L}/{\kappa})$. 

In order to determine the asymptotic behavior of $N_1(d^*,\gamma)$  and $N_2(d^*,\gamma)$ when $d^*$ tends to infinity, 
we will rely on their integral representation through Jacobi's $\theta$-function. Recall that the latter is given
by $\sfh(z)=\sum_{r\in\mathbb{Z}}z^{r^{2}}$, which is well defined for any complex number $z$ belonging to the unit ball 
$|z|<1$. To briefly explain where the relation between $N_i(\gamma)$ and the $\theta$-function comes from,  let 
us denote by $\{a_r\}$ the sequence of coefficients of the power series of $\sfh(z)^{d^*}$, that is $\sfh(z)^{d^*}=
\sum_{r\geq 0} a_r z^r$. One easily checks that $\forall r\in \mathbb{N}$, 
$a_r=\text{Card}\{\bk \in \mathbb{Z}^{d^*} : k_1^{2}+...+k_{d^*}^{2}=r\}$. Thus, for every $\gamma$ such 
that $\gamma d^*$ is integer, we have $N_1(d^*,\gamma)=\sum_{r=0}^{\gamma d^*} a_r$. As a consequence of Cauchy's theorem, 
we get :
$$
N_1(d^*,\gamma)=\frac{1}{2\pi i}\oint \frac{\sfh(z)^{d^*}}{z^{\gamma d^*}} \frac{dz}{z(1-z)}.
$$
where the integral is taken over any circle $|z|=w$ with $0<w<1$. Exploiting this representation and applying the saddle-point 
method thoroughly described in~\cite{Dieudonne}, we get the following result.

\begin{proposition}\label{prop:1}
Let $\gamma>0$ be such that $\gamma d^*$ is an integer and let $\sfl_\gamma(z)=\log{\sfh(z)}-\gamma\log z$. 
\setlength{\parskip}{-2pt}
\begin{enumerate}
\setlength{\parskip}{-4pt}
\item  There is a unique solution $\zb$ in $(0,1)$ to the equation $\sfl_\gamma'(z)=0$. Furthermore, 
the function $\gamma\mapsto z_\gamma$ is increasing and $\sfl_\gamma''(z)>0$.

\item The following equivalences hold true:
\begin{align*}
N_1(d^*,\gamma)&=\bigg(\frac{\sfh(z_\gamma)}{z_\gamma^\gamma}\bigg)^{d^*}\frac{1+o(1)}{\zb(1-\zb)(2\sfl_\gamma''(\zb)\pi d^*)^{1/2}},\\ N_2(d^*,\gamma)&=\bigg(\frac{\sfh(z_\gamma)}{z_\gamma^\gamma}\bigg)^{d^*}\frac{1+o(1)}{\sfh(z_\gamma)\zb(1-\zb)(2\sfl_\gamma''(\zb)\pi d^*)^{1/2}},
\end{align*}
as $d^*$ tends to infinity.
\setlength{\parskip}{3pt}
\end{enumerate}
\end{proposition}
In the sequel, it will be useful to remark that the second part of Proposition~\ref{prop:1} yields
\begin{align}\label{eq:log}
\log \big(N_1(d^*,\gamma)-N_2(d^*,\gamma)\big) & = d^*\sfl_\gamma(z_\gamma)-\frac12\log d^*-
                                         \log\Bigg\{\frac{\sfh(\zb)\zb(1-\zb)(2\sfl_\gamma''(\zb)\pi)^{1/2}}{\sfh(\zb)-1}\Bigg\}+o(1).
\end{align}
In order to get an idea of how the terms $\zb$ and $\sfl_\gamma(\zb)$ depend on $\gamma$, we depicted in Figure~\ref{fig:1} the plots of these quantities as functions of $\gamma>0$.
\begin{figure}[ht]
\includegraphics[width=0.99\textwidth]{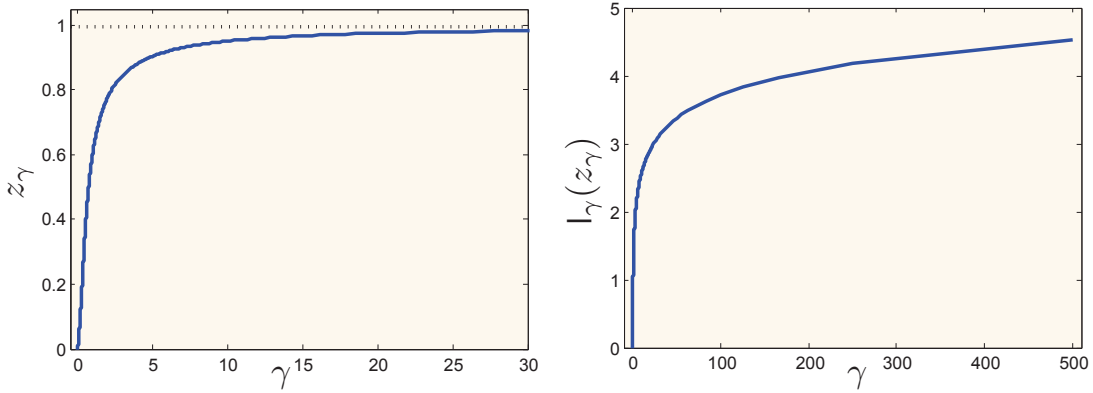}
\vspace{-5pt}
\caption{The plots of mappings $\gamma\mapsto z_\gamma$ and $\gamma\mapsto \sfl_\gamma(\zb)$.}
\label{fig:1}
\end{figure}

\section{Tightness of the assumptions}\label{sec:5}

In this section, we assume that the errors $\varepsilon_i$ are i.i.d.\ Gaussian with zero mean and variance $1$ and 
we focus our attention on the functional class $\widetilde\Sigma(\kappa,L)$ of all functions satisfying assumption 
\textbf{[C2($\kappa,L$)]}.  In order to avoid irrelevant technicalities and to better convey the main results, we assume
that $\kappa=1$ and denote $\widetilde\Sigma_L=\widetilde\Sigma(1,L)$. Furthermore, we will assume that the design 
$\bX_1,\ldots,\bX_n$ is fixed and satisfies 
\begin{equation}\label{orth}
\frac{1}{n}\sum_{i=1}^n\varphi_\bk(\bX_i)\varphi_{\bk'}(\bX_i)\leq \frac{n}{N_1(d^*,L)^2}  
\end{equation}  
for all distinct $\bk$, ${\bk'}\in S_{(d^*L)^{1/2},d^*}\subset\ZZ^d$.  The goal in this section is to provide conditions
under which the consistent estimation of the sparsity support is impossible, that is 
there exists a positive constant $c>0$ and an integer $n_0\in\NN$ such that, if $n\ge n_0$,
$$
\inf_{\widetilde{J}}\sup_{\sff\in \widetilde\Sigma_L} \Pb_\sff(\widetilde{J}\neq J_\sff)\geq c,
$$ 
where the $\inf$ is over all possible estimators of $J_\sff$. To lower bound the LHS of the last inequality, we
introduce a set of $M+1$ probability distributions $\mu_0,\ldots,\mu_{M}$ on $\tilde\Sigma_L$ and use the fact that 
\begin{equation}\label{minor:1}
\inf_{\widetilde{J}}\sup_{\sff\in \widetilde\Sigma_L} \Pb_\sff(\widetilde{J}\neq J_\sff)
\ge \inf_{\widetilde{J}}\frac1{M+1}\sum_{\ell=0}^M\int_{\widetilde\Sigma_L} \Pb_\sff(\widetilde{J}\neq J_\sff)\,\mu_\ell(d\sff).
\end{equation}
These measures $\mu_\ell$ will be chosen in such a way that for each $\ell\ge 1$ there is a set $J_\ell$ of cardinality $d^*$ 
such that $\mu_\ell\{J_\sff=J_\ell\}=1$ and all the sets $J_1,\ldots,J_M$ are distinct. The measure $\mu_0$ is the Dirac measure in 
$0$. Considering these $\mu_\ell$s as ``prior'' probability measures on $\tilde\Sigma_L$ and 
defining the corresponding ``posterior'' probability measures $\PP_0,\PP_1,\ldots,\PP_M$  by 
$$
\PP_\ell(A)=\int_{\tilde\Sigma_L} \Pb_\sff(A)\,\mu_\ell(d\sff),\quad \text{for every measurable set } A\subset\RR^n,
$$ 
we can write the inequality (\ref{minor:1}) as
\begin{equation}\label{minor:2}
\inf_{\widetilde{J}}\sup_{\sff\in \widetilde\Sigma_L} \Pb_\sff(\widetilde{J}\neq J_\sff)
\ge \inf_{\psi}\frac1{M+1}\sum_{\ell=0}^M \PP_\ell(\psi\not=\ell),
\end{equation}
where the $\inf$ is taken over all random variables $\psi$ taking values in $\{0,\ldots,M\}$. The latter $\inf$ will be controlled using 
a suitable version of the Fano lemma, see~\cite{Fano}. In what follows, we denote by $\KL(P,Q)$ the Kullback-Leibler divergence between two
probability measures $P$ and $Q$ defined on the same probability space.
\begin{lemma}[Corollary 2.6 of \cite{Tsybakov09}]\label{lem:fano}
Let $(\mathcal X,\mathcal A)$ be a measurable space and let $P_0,\ldots,P_M$ be probability measures on $(\mathcal X,\mathcal A)$. 
Let us set $\bar{p}_{e,M}=\inf_\psi (M+1)^{-1}\sum_{\ell=0}^M P_\ell\big(\psi\not= \ell\big)$ 
where the $\inf$ is taken over all measurable functions $\psi:\mathcal X\to\big\{0,\ldots,M\big\}$. If for some $0<\alpha< 1$
$$
\frac{1}{M+1}\sum_{\ell=0}^{M}\KL\big(P_\ell,P_0\big)\leq \alpha \log M,
$$
then 
$$ 
\bar{p}_{e,M}\geq \frac{\log (M+1)-\log 2}{\log M}-\alpha.
$$
\end{lemma}

It follows from this lemma that one can deduce a lower bound on $\bar p_{e,M}$, which is the quantity we are interested in, 
from an upper bound on the average Kullback-Leibler divergence between the measures $\PP_\ell$ and $\PP_0$. This roughly means 
that the measures $\mu_\ell$ should not be very far from $\mu_0$ but the probability measures $\mu_\ell$ should be very different 
one from another in  terms of the sparsity pattern of a function $\sff$ randomly drawn according to $\mu_\ell$. This property is 
ensured by the following result.

\begin{lemma}\label{lem:3}
Suppose  $\mu_0=\delta_0$, the Dirac measure at\/ {\sf 0}\/$\in\Sigma_L$. Let $S$ be a subset of\/ $\ZZ^{d}$ of cardinality ${|S|}$ and $A$ 
be a constant.  Define $\mu_S$ as a discrete measure supported on the finite set of functions 
$\{\sff_\oomega = \sum_{\bk\in S}A\omega_\bk\varphi_\bk : \oomega\in\{\pm1\}^S\}$ such that $\mu_S(\sff=\sff_\oomega)=2^{-{|S|}}$ 
for every $\oomega\in\{\pm1\}^S$, \textit{i.e.}, the $\omega_\bk$'s are i.i.d.\  Rademacher random variables under 
$\mu_S$. If, for some $\epsilon\geq 0$, the condition 
$$
\frac{1}{n}\sum_{i=1}^n\varphi_\bk(\bX_i)\varphi_{\bk'}(\bX_i)\leq \epsilon\qquad \forall \bk, {\bk'} \in S
$$
is fulfilled, then
$$
\KL(\PP_1,\PP_0)\le\log\bigg[\int\Big(\frac{d\PP_1}{d\PP_0}\,(\by)\Big)^2\PP_0(d\!\by)\bigg]  \leq 
{4|S|A^4n^2}\Big\{1+\frac{|S|\epsilon}{4nA^2}\Big\}.
$$
\end{lemma}

These evaluations lead to the following theorem, that tells us that the conditions to which we have resorted for proving the consistency in Section~\ref{sec:3} are nearly optimal. 

\begin{theorem}\label{thm2}
Let the design $\bX_1,\ldots, \bX_n\in [0,1]^d$ be deterministic and satisfy (\ref{orth}). 
Let $\gamma^*$ the largest real number such that $d^*\gamma^*$ is integer and 
$L\geq \gamma^*(1+{1}/{2\zbstar})$.  If for some positive number $\alpha<({\log 3-\log 2})/{\log 3}$
\begin{equation}\label{hyp1}
\ds\frac{(N_1(d^*,\gamma^*)-N_2(d^*,\gamma^*))^2\log \binom{d}{d^*}}{n^2N_1(d^*,\gamma^*)}\geq \frac{\alpha}{5},
\end{equation}
then there exists a positive constant $c>0$ and a $d_0\in\NN$ such that, if $d^*\ge d_0$,
$$
\inf_{\widetilde{J}}\sup_{\sff\in \widetilde\Sigma_L} \Pb_\sff(\widetilde{J}\neq J_\sff)\geq c.
$$ 
\end{theorem}

\begin{proof}
We apply the Fano lemma with $M = \binom{d}{d^*}$. We choose $\mu_0,\ldots, \mu_M$ as follows. 
$\mu_0$ is the Dirac measure $\delta_0$, $\mu_1$ is defined as in Lemma~\ref{lem:3} with $S=\mathcal{C}_1(d^*,\gamma^*)$ and $A={\big[N_1(d^*,\gamma^*)-N_2(d^*,\gamma^*)\big]}^{-1/2}$. The measures $\mu_2,\ldots,\mu_M$ 
are defined similarly and correspond to the $M-1$ remaining sparsity patterns of cardinality $d^*$. 

In view of inequality (\ref{minor:2}) and Lemma~\ref{lem:fano}, it suffices to show that the measures $\mu_\ell$ 
satisfy $\mu_\ell(\widetilde\Sigma_L)=1$ and $\sum_{\ell=0}^M\KL(\PP_\ell,\PP_0)\le (M+1)\alpha \log M$. 
Combining Lemma~\ref{lem:3} with $|S|=N_1(d^*,\gamma^*)$ and condition~(\ref{orth}), one easily checks that 
equation (\ref{hyp1}) implies the desired bound on $\sum_{\ell=0}^M\KL(\PP_\ell,\PP_0)$. 

Let us show now that $\mu_1(\widetilde\Sigma_L)=1$. By symmetry, this will imply that $\mu_\ell(\widetilde\Sigma_L)=1$ for
every $\ell$. Since $\mu_1$ is supported by the set $\{\sff_\oomega:\oomega\in \{\pm1\}^{\mathcal C(d^*,\gamma^*)}\}$, it is 
clear that 
$$
\sum_{k_1\neq 0} \theta_\bk^2[\sff_\oomega]=A^2[N_1(d^*,\gamma^*)-N_2(d^*,\gamma^*)]=1
$$
{and}, for every $j=1,\ldots,d^*$, 
$$ 
\sum_{\bk\in \ZZ^d} k_j^2\theta_\bk^2[\sff_\oomega]=\sum_{\bk\in \mathcal C(d^*,\gamma^*)} k_j^2 A^2=\frac1{d^*} 
\sum_{j=1}^{d^*}\sum_{\bk\in \mathcal C(d^*,\gamma^*)} k_j^2 A^2\leq A^2 \gamma^*
N_1(d^*,\gamma^*). 
$$
By virtue of Proposition~\ref{prop:1}, as $d^*$ tends to infinity, ${N_1(d^*,\gamma^*)}/{N_2(d^*,\gamma^*)}$ is asymptotically equivalent to $\sfh(\zbstar)>1+2\zbstar$. Hence, for $d^*$ large enough, 
$$
A^2N_1(d^*,\gamma^*)=\frac{N_1(d^*,\gamma^*)}{N_1(d^*,\gamma^*)-N_2(d^*,\gamma^*)}<\frac{1}{2\zbstar}+1.
$$
As a consequence, for every $j=1,\ldots,d^*$,
$$
\sum_{\bk\in \ZZ^d} k_j^2\theta_\bk^2[\sff_\oomega]\le \gamma^*\Big(\frac{1}{2\zbstar}+1\Big)\le L,
$$
where the last inequality follows from the definition of $\gamma^*$.
\end{proof}

Note that Theorem~\ref{thm2} is concerned by the case where the intrinsic dimension is not too small, which is
the most interesting case in the present context. However, a much simpler result can be established showing that
the conditions of Theorem~\ref{thm1} are tight in the case of fixed intrinsic dimension as well. 

\begin{proposition}\label{prop:2}
Let the design $\bX_1,\ldots, \bX_n\in [0,1]^d$ be either deterministic or random. If for some  
positive $\alpha<{(\log 3-\log 2)}/{\log 3}$, the inequality 
$$
\frac{d^*\big(\log d-\log d^*\big)}{n}\geq \alpha^{-1}
$$ 
holds true, then there is a constant $c>0$ such that $\inf_{\widetilde J_n}\sup_{f\in \widetilde\Sigma_L} 
\Pb_\sff(\widetilde J_n\neq J_\sff)\geq c$.
\end{proposition}

 \section{Discussion}\label{sec:6}

The results proved in previous sections almost exhaustively answer the questions on the existence of consistent 
estimators of the sparsity pattern in the problem of nonparametric regression. In fact as far as only rates of convergence
are of interest, the result obtained in Theorem~\ref{thm1} is shown in Section~\ref{sec:5} to be unimprovable. Thus only 
the problem of finding sharp constants remains open. To make these statements more precise, let us consider the simplified 
set-up $\sigma=\kappa=1$ and define the following two regimes: 
\begin{enumerate}
\item[$\checked$] The regime of fixed sparsity, \textit{i.e.}, when the sample size $n$ and the ambient dimension $d$ 
tend to infinity but the intrinsic dimension $d^*$ remains constant or bounded. 
\item[$\checked$] The regime of increasing sparsity, \textit{i.e.}, when the intrinsic dimension $d^*$ tends to infinity 
along with the sample size $n$ and the ambient dimension $d$. For simplicity, we will assume that $d^*=O(d^{1-\epsilon})$ for some 
$\epsilon>0$.
\end{enumerate}
In the fixed sparsity regime, in view of Theorem~\ref{thm1}, consistent estimation of the sparsity pattern can be achieved using
the estimator $\widehat J$ as soon as $(\log d)/n \le c_\star$, where $c_\star$ is the constant defined by
$$
c_\star=\min\Big(\frac{L_2^2}{2d^*L_\infty^2}, \frac{\gmin^2}{2^8(1+L_2)^2d^*N(d^*,2L)}\Big).
$$
This follows from the fact that the tuning parameter $m$ is fixed and that the probability of the error, bounded by $3(6md)^{d^*}$ 
tends to zero as $d\to\infty$. On the other hand, by virtue of Proposition~\ref{prop:2}, consistent estimation of the sparsity pattern 
is impossible if $(\log d)/n> c^\star$, where $c^\star=2\log 3/(d^*\log(3/2))$. Thus, up to multiplicative constants $c_\star$ and
$c^\star$ (which are clearly not sharp), the result of Theorem~\ref{thm1} can not be improved. 

In the regime of increasing sparsity, the second inequality in (\ref{cond3}) is the most stringent one. Taking the logarithm of both sides
and using formula (\ref{eq:log}) for $N(d^*,2L)=N_1(d^*,2L)-N_2(d^*,2L)$, we see that consistent estimation of $J$ is possible when 
\begin{align}\label{eq:10}
\underline{c}_1d^*+\frac12\log d^*+\log\log d-\log n<\underline{c}_2,
\end{align}
with $\underline{c}_1=\sfl_{2L}(z_{2L})$ and $\underline{c}_2=2(\log(\gmin)-\log(17(\sigma+L_2))+
\log\Big\{\frac{\sfh(z_{2L})z_{2L}(1-z_{2L})(2\sfl_{2L}''(z_{2L})\pi)^{1/2}}{\sfh(z_{2L})-1}\Big\}$. 
On the other hand, by virtue of (\ref{eq:log}), $\log\Big\{\frac{[N_1(d^*,\gamma)-N_2(d^*,\gamma)]^2}{N_1(d^*,\gamma)}\Big\}=
d^*\sfl_\gamma(\zb)-\frac12\log d^*-\log\Big\{\frac{\sfh(\zb)^2\zb(1-\zb)(2\sfl_\gamma''(\zb)\pi)^{1/2}}{(\sfh(\zb)-1)^2}\Big\}+o(1)$. 
Therefore, Theorem~\ref{thm2} yields that it is impossible to consistently estimate $J$ if 
\begin{align}\label{eq:11}
\bar{c}_1d^*+\frac12\log d^*+\log\log d-2\log n>\bar{c}_2,
\end{align}
where $\bar{c}_1=\sfl_{\gamma^*}(z_{\gamma^*})$ and 
$\bar c_2=\log\Big\{\frac{\sfh(z_{\gamma^*})^2z_{\gamma^*}(1-z_{\gamma^*})(2\sfl_{\gamma^*}''(z_{\gamma^*})\pi)^{1/2}}{(\sfh(z_{\gamma^*})-1)^2}\Big\}+
\log\log(3/2)-\log 5-\log\log 3$. A very simple consequence of inequalities (\ref{eq:10}) and (\ref{eq:11}) is that 
the consistent recovery of the sparsity pattern is possible under the condition $d^*/\log n\to 0$ and impossible for $d^*/\log n\to \infty$ as
$n\to\infty$, provided that $\log\log d=o(\log n)$.

Let us stress now that, all over this work, we have deliberately avoided any discussion on the computational aspects of the variable selection 
in nonparametric regression. The goal in this paper was to investigate the possibility of consistent recovery without paying attention to the 
complexity of the selection procedure. This lead to some conditions that could be considered a benchmark for assessing the properties of 
sparsity pattern estimators. As for the estimator proposed in Section~\ref{sec:3}, it is worth noting that its computational complexity is not 
always prohibitively large. A recommended strategy is to compute the coefficients $\hatthetak$ in a stepwise manner; at each step 
$K=1,2,\ldots,d^*$ only the coefficients $\hatthetak$ with $\|\bk\|_0=K$ need to be computed and compared with the threshold. If some 
$\hatthetak$ exceeds the threshold, then all the covariates $X^j$ corresponding to nonzero coordinates of $\bk$ are considered as relevant. 
We can stop this computation as soon as the number of covariates classified as relevant attains $d^*$. While the worst-case complexity of this procedure is exponential, there are many functions $\sff$ 
for which the complexity of the procedure will be polynomial in $d$. For example, this is the case for additive models in which 
$\sff(\bx)= \sff_1(x_{i_1})+\ldots+\sff_{d^*}(x_{i_{d^*}})$ for some univariate functions $\sff_1,\ldots,\sff_{d^*}$.

%
%
%
%
%

\bibliography{Laetitia_COLT}

\appendix
\section{Proof of Theorem~\ref{thm1}}

The empirical Fourier coefficients can be decomposed as follows: 
\begin{equation}
\hatthetak=\tilde{\theta}_\bk+z_\bk,
\qquad\text{where}\qquad  \tilde{\theta}_\bk=\frac{1}{n}\sum_{i=1}^n  \frac{\varphi_\bk(\bX_i)}{\sfg(\bX_i)}\sff(\bX_i)
\qquad  \text{and}\qquad  z_\bk=\frac{\sigma}{n}\sum_{i=1}^n  \frac{\varphi_\bk(\bX_i)}{\sfg(\bX_i)}\varepsilon_i.
\end{equation}
If, for a multi index $\bk$, $\theta_\bk=0$, then   the corresponding empirical Fourier coefficient will be close to zero 
with high probability. To show this, let us first look at what happens with $z_\bk$'s.
We have, for every real number $x$,
$$
\Pb\big(|z_\bk|>x\,\big|\,\bX_1,\ldots,\bX_n\big)\leq \exp\Big(-\frac{x^2}{2\sigma_\bk^2}\Big) \quad \forall \bk\in S_{m,d^*}
$$ 
with
$$
\ds \sigma_\bk^2=\frac{\sigma^2}{n^2}\sum_{i=1}^n \frac{\varphi_\bk(\bX_i)^2}{\sfg(\bX_i)^2}\leq \frac{2 \sigma^2}{\gmin^2n}.
$$
Therefore, for every $\bk\in S_{m,d^*}$, it holds that $\Pb\big(|z_\bk|>x|\bX_1,\ldots,\bX_n\big)\leq \exp({-n\gmin^2x^2}/{4\sigma^2})$. 
This entails that by setting $\lambda_1=({8\sigma^2d^*\log(6md)}/{n\gmin^2})^{1/2}$ and by using the inequalities
\begin{align*}
\text{Card}(S_{m,d^*})&=\sum_{i=0}^{d^*}\binom{d}{i} (2m)^i
                      \le (2m)^{d^*}\sum_{i=0}^{d^*} \frac{d^i}{i!}\\
                      &\le 3 (2md)^{d^*}
										  \le (6md)^{d^*},
\end{align*}
we get
\begin{align*}
\Pb\Big(\max_{\bk\in S_{m,d^*}}|z_\bk|>\lambda_1\,|\,\bX_1,\ldots,\bX_n\Big)&\leq \sum_{\bk\in S_{m,d^*}}
\Pb\Big(|z_\bk|>\lambda_1\,|\,\bX_1,\ldots,\bX_n\Big)\\
&\leq \text{Card}(S_{m,d^*})e^{-n\gmin^2\lambda_1^2/{4\sigma^2}} \leq (6md)^{-d^*}.
\end{align*}
Next, we use a concentration inequality for controlling large deviations of $\tilde{\theta}_\bk$'s from $\theta_\bk$'s. 
Recall that in view of the definition $\tilde{\theta}_\bk=\frac{1}{n}\sum_{i=1}^n \frac{\varphi_\bk(\bX_i)}{\sfg(\bX_i)}\sff(\bX_i)$, we have $\E(\tilde{\theta}_\bk)=\theta_\bk$.  
By virtue of the boundedness of $\sff$, it holds that $|\frac{\varphi_\bk(\bX_i)}{\sfg(\bX_i)}\sff(\bX_i)|\leq { \sqrt{2}L_\infty}/{\gmin}$. 
Furthermore, the bound $V\triangleq\text{Var}\big(\frac{\varphi_\bk(\bX_i)}{\sfg(\bX_i)} \sff(\bX_i)\big)\leq\int f^2(\bx)\frac{\varphi_\bk^2(\bx)}{\sfg(\bx)}d\bx\leq {2L_2^2}/{\gmin^2}$ combined with Bernstein's inequality yields
\begin{align*}
\Pb\big(|\tilde{\theta}_\bk-\theta_\bk|>t\big)&\leq 2\exp\Big(-\frac{nt^2}{2(V+{t\sqrt{2}L_\infty/3\gmin})}\Big)\\
&\leq 2\exp\Big(-\frac{\gmin^2 nt^2}{4L_2^2+{tL_\infty}\gmin}\Big),\qquad \forall t>0.
\end{align*}
Let us define $\lambda_2=4L_2\Big({\frac{d^*\log(6md)}{n\gmin^2}}\Big)^{1/2}$. Then,  
$$
\Pb\big(|\tilde{\theta}_\bk-\theta_\bk|>\lambda_2\big)\leq 2\exp\bigg(-\frac{4L_2^2d^*\log(6md)}
{L_2^2+L_\infty L_2\big({\frac{d^*\log(6md)}{n}}\big)^{1/2}  }\bigg).
$$
The first inequality in condition (\ref{cond3}) implies that the denominator in the exponential is not larger than  
$2L_2^2$. Hence,
$$
\Pb\Big(\max_{\bk\in S_{m,d^*}}|\tilde{\theta}_\bk-\theta_\bk|>\lambda_2\Big)\leq 2/{(6md)^{d^*}}.
$$
Let $\mathcal{A}_1=\big\{\max_{\bk\in S_{m,d^*}}|z_\bk|\leq\lambda_1\big\}$ and 
$\mathcal{A}_2=\big\{\max_{\bk\in S_{m,d^*}}|\tilde{\theta}_\bk|\leq \lambda_2\big\}$.  One easily checks that
$$
\Pb\big(J^c\not\subset \widehat{J}^c\big)\leq \Pb\big(\mathcal{A}_1^c\big)+\Pb\big(\mathcal{A}_2^c\big)\le 3/(6md)^{d^*}.
$$
As for the converse inclusion, we have 
\begin{align*}
\Pb(J\not\subset \widehat J)&\leq\Pb\Big(\exists j\in J\ \text{s.t.}\ \max_{\bk\in S_{m,d^*}:\,k_j\not=0} |\hatthetak|\le{\lambda}\Big)\\
&\le \1\Big\{\exists j\in J\ \text{s.t.}\ \max_{\bk\in S_{m,d^*}\,:k_j\not=0}|\theta_\bk|\le 2 \lambda \Big\}+\Pb\big(\mathcal A_1^c\big)
+\Pb\big(\mathcal A_2^c\big).
\end{align*} 
We show now that the first term in the last line is equal to zero. If this was not the case, then for some value $j_0$ we would 
have $Q_{j_0}\ge \kappa$ and $|\theta_\bk|\le 2\lambda$, for all $\bk\in S_{m,d^*}$ such that $k_{j_0}\not=0$. This would imply that
$$
Q_{j_0,m,d^*}{\triangleq}\sum_{\bk\in S_{m,d^*}\,:k_{j_0}\not=0}  \theta_\bk^2\le 4\lambda^2 N(d^*,^2L/\kappa). 
$$ 
On the other hand,
\begin{align*}
Q_{j_0}-Q_{j_0,m,d^*}&\le\sum_{\|\bk\|_{2}\ge m}  \theta_\bk^2\le m^{-2}
\sum_{\|\bk\|_{2}\ge m} \sum_{j\in J}k_{j}^{2} \theta_\bk^2
\le \frac{L d^*}{ m^{2}}.
\end{align*}
Remark now that the choice of the truncation parameter $m$ proposed in the statement of the proposition implies that 
$Q_{j_0}-Q_{j_0,m,d^*}\le \kappa/2$. Combining these estimates, we get
$
Q_{j_0}\le \frac{\kappa}2+ 4\lambda^2N(d^*,m^2/d^*),
$
which is impossible since $Q_{j_0}\ge \kappa$.

\section{Proof of Proposition~\ref{prop:1}}
\paragraph{Proof of the first assertion.} This proof can be found in~\cite{Mazo}, we repeat here the arguments therein for the sake of keeping
the paper self-contained. Recall that $N_1(d^*,\gamma)$ admits an integral representation with the integrand: 
$$
\frac{\sfh(z)^{d^*}}{z^{\gamma d^*}} \frac{1}{z(1-z)}=\frac{1}{z(1-z)}\exp \bigg[d^*\log\bigg(\frac{\sfh(z)}{z^\gamma}\bigg)\bigg].
$$
For any real number $y>0$, we define $\phi(y)=e^{-y}\sfh'(e^{-y})/\sfh(e^{-y})=\sum_{k=-\infty}^{k=+\infty} k^2 e^{-yk^{2}}/\sum_{k=-\infty}^{k=+\infty} e^{-yk^{2}}$ in such a way that 
$$
\phi(y)=\gamma \quad\Longleftrightarrow\quad \frac{\sfh'(e^{-y})}{\sfh(e^{-y})}=\frac{\gamma }{e^{-y}}
\quad\Longleftrightarrow\quad \sfl_\gamma'(e^{-y})=0.
$$ 
By virtue of the Cauchy-Schwarz inequality, it holds that
$$
\sum k^{4}e^{-yk^{2}}\sum e^{-yk^{2}}>\Big(\sum k^{2}e^{-yk^{2}}\Big)^2, \quad \forall y\in(0,\infty),
$$
implying that $\phi'(y)<0$ for all $y\in (0,\infty)$, \textit{i.e.}, $\phi$ is strictly decreasing. Furthermore, $\phi$ is 
obviously continuous with $\lim_{y\to 0}\phi(y)=+\infty$ and $\lim_{y\to\infty} \phi(y)=0$. These properties imply the 
existence and the uniqueness of $y_\gamma\in(0,\infty)$ such that ${\phi(y_\gamma)}=\gamma$. Furthermore, as the inverse
of a decreasing function, the function $\gamma\mapsto y_\gamma$ is decreasing as well. We set $z_\gamma=e^{-y_{\gamma}}$ 
so that $\gamma\mapsto z_\gamma$ is increasing.

We also have  
\begin{align*}
\sfl_\gamma''(\zb) &= \frac{\sfh''\sfh-(\sfh')^2}{\sfh^2}(\zb)+\frac{\gamma}{\zb^2}=z_\gamma^{-2}\bigg\{
\frac{\sum_k (k^4-k^2)\zb^{k^2}}{\sum_k\zb^{k^2}}-\bigg(\frac{\sum_k k^2\zb^{k^2}}{\sum_k\zb^{k^2}}\bigg)^2+\gamma\bigg\}\\
&=\zb^{-2}\big\{-\phi'(y_\gamma)-\phi(y_\gamma)+\gamma\big\}=-\zb^{-2}\phi'(y_\gamma)>0.
\end{align*}

\paragraph{Proof of the second assertion.}
We apply the saddle-point method to the integral representing  $N_1$ see, \textit{e.g.}, Chapter IX in \cite{Dieudonne}. 
It holds that
\begin{equation}
N_1(d^*,\gamma) = \frac{1}{2\pi i}\oint_{|z|=\zb} \frac{\sfh(z)^{d^*}}{z^{\gamma d^*}} \frac{dz}{z(1-z)}
                = \frac{1}{2\pi i}\oint_{|z|=\zb} \{z(1-z)\}^{-1}e^{d^*\sfl_\gamma(z)}{dz}.
\end{equation}
The first assertion of the proposition provided us with a real number $\zb$ such that $\sfl_\gamma'(\zb)=0$ et $\sfl_\gamma''(\zb)>0$. 
The  tangent to the steepest descent curve at $\zb$ is vertical. The path we choose for integration is the circle with center 0 and 
radius $\zb$.  As this circle and the steepest descent curve  have the same tangent at $\zb$, applying formula (1.8.1) of 
\cite{Dieudonne} (with $\alpha=0$ since $\sfl''(\zb)$ is real and positive), 
we get that  
$$
\frac{1}{2\pi i}\oint_{|z|=\zb} \{z(1-z)\}^{-1}e^{d^*\sfl_\gamma(z)}{dz}=
\frac{1}{2\pi i}\sqrt{\frac{2\pi}{d^*\sfl_\gamma''(\zb)}}e^{{\rm i\pi/2}}\{\zb(1-\zb)\}^{-1}e^{d^*\sfl_\gamma(\zb)}(1+o(1)),
$$
when $d^*\to\infty$, as soon as the condition\footnote{$\Re u$ stands for the real part of the complex number $u$.}  
$\Re [\sfl_\gamma(z)-\sfl_\gamma(\zb)]\leq -\mu$
is satisfied for some $\mu>0$ and for any $z$ belonging to the circle $|z|=|\zb|$ and lying not too close to $\zb$. 
To check that this is indeed the case, we remark that $\Re [\sfl_\gamma(z)]=\log \big|\frac{\sfh(z)}{z^\gamma}\big|$.   
Hence, if  $z=\zb e^{{\rm i}\omega}$ with $\omega\in[\omega_0,2\pi-\omega_0]$ for some $\omega_0\in]0,\pi[$, then
$$
\Big|\frac{\sfh(z)}{z^\gamma}\Big|=\frac{|1+2z+2\sum_{k> 1} z^{k^{2}}|}{\zb^\gamma}
\le\frac{|1+z|+\zb+2\sum_{k>1}\zb^{k^{2}}}{\zb^\gamma} \le\frac{|1+e^{{\rm i}\omega_0}\zb|+\zb+2\sum_{k>1}\zb^{k^{2}}}{\zb^\gamma}.
$$
Therefore $\Re[\sfl_\gamma(z)-\Re \sfl_\gamma(\zb)]\leq -\mu$ with 
$\mu=\log \Big(\frac{1+2\zb+\sum_{k\geq 1}\zb^{k^{2}}}{|1+\zb e^{{\rm i}\omega_0}|+\zb+\sum_{k\geq 1}\zb^{k^{2}}}\Big)>0$.
This completes the proof for the term $N_1(d^*,\gamma)$.  The term $N_2(d^*,\gamma)$ can be dealt in the same way.

\section{Proof of Lemma~\ref{lem:3}}

Let $\phi(\cdot)$ be the density of $\mathcal{N}(0,1)$ and let
$$
p_\sff(\by)\triangleq\prod_{i=1}^n \phi\Big(y_i-\sff(\bX_i)\Big),\qquad\forall \by\in\RR^n.
$$
Since the errors $\varepsilon_i$ are Gaussian, the posterior probabilities  $\PP_0$ and $\PP_1$  are absolutely 
continuous w.r.t.\ the Lebesgue measure on  $\RR^n$ and admit the densities   
$$
p_0(\by)=\prod_{i=1}^n \phi(y_i),\qquad\text{and}\qquad
p_1(\by)=\Ex_{\sff\sim \mu_S}p_\sff(\by),\qquad\forall \by\in\RR^n.
$$ 
Simple algebra yields:
$$
p_\sff(\by)=C_\sff p_0(\by)\prod_{i=1}^n 
\exp\Big\{y_i\sff(\bX_i)\Big\},\quad\forall\by\in\RR^n,
$$ 
where $C_\sff=\prod_{i=1}^n \exp\big\{-\sff(\bX_i)^2/2\big\}$. Thus, 
$$
\frac{p_1}{p_0}(\by)=\Ex_{\sff\sim\mu_S} \Big[C_\sff\prod_{i=1}^n \exp\Big\{{y_i}\sff(\bX_i)\Big\}\Big].
$$
Therefore,
\begin{align*}
 \int_{\RR^n} \Big(\frac{p_1}{p_0}(\by)\Big)^2p_0(\by)d\by&=\Ex_{(\sff,\sff')\sim\mu_S\otimes\mu_S}\Big[C_\sff C_{\sff'} 
 \int_{\RR^n}\prod_{i=1}^n\bigg( \exp\Big\{{y_i}(\sff+\sff')(\bX_i)\Big\}\phi(y_i)\bigg)d\by\Big] \\
 &=\Ex_{(\sff,\sff')\sim\mu_S\otimes\mu_S}\Big[C_\sff C_{\sff'} \prod_{i=1}^n \exp\Big(\frac{1}{2}(\sff+\sff')^2(\bX_i)\Big)\Big] \\
&=\Ex_{(\sff,\sff')\sim\mu_S\otimes\mu_S}\Big[ \exp\Big(\sum_{i=1}^n\sff(\bX_i)\sff'(\bX_i)\Big)\Big]\\
&=\frac1{2^{2|S|}}\sum_{\oomega,\oomega'\in\{\pm1\}^S}\prod_{\bk, {\bk'} \in S } \exp\Big(\omega_\bk\omega_{\bk'}'b_{\bk{\bk'}}\Big),
\end{align*}
where $b_{\bk{\bk'}}={A^2}\sum_{i=1}^n\varphi_\bk(\bX_i)\varphi_{\bk'}(\bX_i)$, for all $ \bk, {\bk'} \in S$.  Note that  
$0\le b_{\bk\bk}\leq {2A^2n}$ and $|b_{\bk{\bk'}}|\leq {A^2n\epsilon}$, for all $ \bk,\bk' \in S$ such that $\bk'\not=\bk$. Now, 
on the one hand,  for a fixed pair $(\oomega,\oomega')$, we have 
$$
\prod_{\bk\neq{\bk'}  } \exp\Big(\omega_\bk\omega_{\bk'}'b_{\bk{\bk'}}\Big)\leq \exp\big({|S|}^2A^2n\epsilon\big).
$$
On the other hand, if we  are given a sequence of numbers $(b_{\bk\bk})$ indexed by $S$, we have 
$$
\frac1{2^{2|S|}}\sum_{\oomega,\oomega'}\prod_{\bk \in S}e^{\omega_\bk\omega_\bk'b_{\bk \bk}}= 
\prod_{\bk \in S}\frac{e^{b_{\bk\bk}}+e^{-b_{\bk\bk}}}{2}\leq 
\prod_{\bk \in S}e^{b_{\bk\bk}^2}\le \exp\Big({4|S|A^4n^2}\Big).
$$
From these remarks it results that 
\begin{align*}
\int_\RR^d \Big(\frac{p_1}{p_0}\;(\by)\Big)^2p_0(\by)d\!\by  \leq \exp\Big({4|S|A^4n^2}\Big\{1+\frac{|S|\epsilon}{4nA^2}\Big\}\Big),
\end{align*}
and the claim of the lemma follows.

\section{Proof of Proposition~\ref{prop:2}}
Let $M=\binom{d}{d^*}$ and let $\{\sff_0,\sff_1, \ldots, \sff_M\}$ be a set included in $\widetilde\Sigma_L$. 
Let $I_1,\ldots,I_M$ be all the subsets of $\{1,\ldots,d\}$ containing exactly $d^*$ elements somehow enumerated. 
Let us set $\sff_0\equiv 0$ and define $\sff_\ell$, for $\ell\neq 0$, by its Fourier coefficients 
$\{\theta^\ell_\bk:\bk\in\ZZ^d\}$ as follows:  
$$
\theta^\ell_{\bk}=\begin{cases}
1, & \bk=(k_1,\ldots,k_d)=(\1_{1\in I_\ell},\ldots,\1_{d\in I_\ell}),\\
0, & \text{otherwise}.
\end{cases}
$$
Obviously, all the functions $\sff_\ell$ belong to $\Sigma$ and, moreover, each $\sff_\ell$ has $I_\ell$ as sparsity pattern. 
One easily checks that our choice of $\sff_\ell$ implies $\KL(\Pb_{\sff_\ell},\Pb_{\sff_0})=n\|\sff_\ell-\sff_0\|_2^2=n$. Therefore, 
if $\alpha\log M =\alpha\log \binom{d}{d^*}\ge n$, the desired inequality is satisfied. To conclude it suffices 
to note that $\log\binom{d}{d^*}$ is larger than or equal to $d^*\log(d/d^*)=d^*\big(\log d-\log d^*\big)$.
\end{document}